\documentclass[11pt]{article}

\usepackage[margin=1in]{geometry}
\usepackage{amsmath,amssymb,amsthm,mathtools}
\usepackage{hyperref}
\usepackage{enumitem}
\usepackage{comment}
\usepackage{tikz}

\usepackage{xcolor}

\newtheorem{theorem}{Theorem}[section]
\newtheorem{lemma}[theorem]{Lemma}
\newtheorem{proposition}[theorem]{Proposition}
\newtheorem{corollary}[theorem]{Corollary}
\newtheorem{fact}[theorem]{Fact}
\theoremstyle{definition}
\newtheorem{definition}[theorem]{Definition}

\theoremstyle{remark}
\newtheorem{remark}[theorem]{Remark}

\newcommand{\F}{\mathbb{F}}
\newcommand{\PP}{\mathbb{P}}
\newcommand{\AG}{\mathrm{AG}}

\title{Grid-free linear hypergraphs via Cayley-Bacharach}
\author{%
Cosmin Pohoata\thanks{Department of Mathematics, Emory University, Atlanta, GA. Email: {\tt cosmin.pohoata@emory.edu}. Research supported by NSF grant DMS-2246659. }%
}
\date{}

\begin{document}
\maketitle

\begin{abstract}
We give a new construction showing that for every $r\ge 3$, there exists an $r$-uniform \emph{linear} hypergraph on $n$ vertices with $\Theta_r(n^2)$ edges and no copy of the $r\times r$ grid.

This complements the works of F\"uredi--Ruszink\'o, Glock--Joos--Kim--K\"uhn--Lichev, Delcourt--Postle for $r \geq 4$, as well as the subsequent constructions of Gishboliner--Shapira and Solymosi for the case $r=3$.
\end{abstract}

\section{Introduction} \label{sec:intro}

An $r$-uniform hypergraph $H=(V,E)$ is \emph{linear} if any two distinct edges intersect in at most one vertex.
Linearity forces a universal quadratic upper bound on the number of edges: if $|V|=n$, then each edge contributes
$\binom{r}{2}$ distinct unordered vertex pairs, and no pair can occur in two different edges. Hence
\begin{equation}\label{eq:paircount}
|E|
\le \frac{\binom{n}{2}}{\binom{r}{2}}
= \frac{n(n-1)}{r(r-1)}.
\end{equation}
When equality holds, every pair of vertices lies in a unique edge and $H$ is a Steiner system $S(2,r,n)$.
A basic theme in extremal hypergraph theory and design theory is to understand which \emph{local} configurations can be
avoided while retaining near-Steiner density. This includes, for instance, the recent surge of work on high-girth
Steiner systems and partial designs
(e.g.\ \cite{HighGirth,conflictfree} and references therein).

The focus here will be on the case of the \emph{$r\times r$ grid} $G_{r\times r}$, which has also received quite a bit of attention in the last few years. This is the $r$-uniform linear hypergraph on vertex set
\[
\{v_{ij}: 1\le i,j\le r\},
\]
and with $2r$ edges given by the $r$ \emph{rows} $\{v_{ij}\}_{j=1}^r$ and the $r$ \emph{columns} $\{v_{ij}\}_{i=1}^r$.
For convenience, write $\mathrm{ex}_{\mathrm{lin}}(n,G_{r\times r})$ for the maximum number of edges in an $r$-uniform linear hypergraph on $n$ vertices containing no copy of $G_{r\times r}$.
The systematic study of $\mathrm{ex}_{\mathrm{lin}}(n,G_{r\times r})$ was initiated by F\"uredi and Ruszink\'o~\cite{FurediRuszinko}. For $r\ge4$ they constructed an approximate Steiner system, showing that
\[
\mathrm{ex}_{\mathrm{lin}}(n,G_{r\times r})
=
\frac{n(n-1)}{r(r-1)}-O_r(n^{8/5}).
\]
So in particular one can obtain $G_{r\times r}$-free $r$-uniform linear hypergraphs while losing only a lower-order term from the total edge budget \eqref{eq:paircount}. Their approach has two conceptual layers:
\begin{itemize}[leftmargin=2em]
\item
A very explicit \emph{partite} ``line model'' over $\F_q$, with vertex set $V=[r]\times \mathbb{Z}_q$ and edges
\[
A(y,m)=\{(1,y),(2,y+m),\dots,(r,y+(r-1)m)\}\qquad (y,m\in\mathbb{Z}_q),
\]
which can be viewed as discretized families of parallel lines (``slopes'' $m$) across $r$ vertical fibers.
For $r\ge4$ one can take essentially all slopes (after an appropriate simultaneous diophantine approximation step) and obtain an $r$-partite linear $r$-graph with $\Theta(q^2)$ edges and no $r\times r$ grid
(see \cite[\S3.1--3.2]{FurediRuszinko}).
\item
An inductive ``packing'' step: place additional dense grid-free linear $r$-graphs inside the parts and use random
permutations to keep the number of newly created grids small, then delete one edge from each remaining grid.
This yields the near-Steiner density for all large $n$.
\end{itemize}
For $r \geq 4$, different constructions (as well as generalizations) also appear in work of Glock--Joos--Kim--K\"uhn--Lichev~\cite[Theorem~1.6]{conflictfree} on conflict-free hypergraph matchings, as well as in the work of Delcourt--Postle \cite{DP} on refined absorption. Rather remarkably, for $r=3$, it turns out however that none of these different methods can generate dense $G_{3 \times 3}$-free linear hypergraphs. For example, in the partite line model of F\"uredi and Ruszink\'o one may easily find $3\times3$ grids, so in order to avoid this issue one has to use a much smaller (and carefully chosen) set of slopes. This led F\"uredi and Ruszink\'o only to a lower bound of the form $\mathrm{ex}_{\mathrm{lin}}(n,G_{3\times 3}) = \Omega(n^{1.8})$. 

In \cite{GishShapira}, Gishboliner and Shapira subsequently gave the first genuinely quadratic construction of a
linear $3$-uniform hypergraph with no $3\times3$ grid, establishing that
$$\mathrm{ex}_{\mathrm{lin}}(n,G_{3\times 3}) = \left(\frac{1}{16}-o(1)\right)n^2.$$
Their construction comes from a $3$-partite hypergraph, with vertices inside $\mathbb{F}_{p} \times \mathbb{F}_{p} \times \mathbb{F}_{p}$ and edges of the form $\left\{(x,\, x+a,\, x \cdot a):\ x \in X,\ a \in A\right\}$, where $X$ is the quadratic residues and $A$ the set of quadratic non-residues in $\mathbb{F}_p$. Grid-freeness then follows from a rather striking calculation: any hypothetical $G_{3 \times 3}$ forces three elements $p_1, p_2, p_3 \in X$ to satisfy $(p_1 - p_2)(p_2 - p_3)(p_3 - p_1) = 0$, contradicting the fact that grid vertices are distinct.
In \cite{Solymosi}, Solymosi later provided an alternative construction using two conics in the affine plane $\mathbb{F}_{q}^{2}$, improving also the leading constant in lower bound to $\mathrm{ex}_{\mathrm{lin}}(n,G_{3\times 3}) = \left(\frac{1}{12}-o(1)\right)n^2$. An appealing feature of the construction from \cite{Solymosi} is that the $G_{3 \times 3}$-freeness can be deduced directly from the classical Pascal theorem, instead of (somewhat) mysterious algebraic identities. 

From a modern viewpoint, Pascal's theorem is a special case of the Cayley--Bacharach theorem, which asserts that for a complete intersection of plane curves, a low-degree curve cannot pass through all but one of the intersection points without passing through the last as well. See for example the survey of Eisenbud-Green-Harris \cite{Eisenbud} or Tao's blog \cite{TaoBlog} for wonderful accounts of this story, as well as Lemma \ref{cor:planeCB} below which will contain the precise statement. In this note, we leverage this connection one step further, by using the Cayley-Bacharach theorem to show the following general result:

\begin{theorem}\label{thm:main}
For every $r \geq 3$ and every sufficiently large odd prime power $q$, there exists an $r$-uniform linear
hypergraph $H_{r,q}$ with no $r\times r$ grid, with
\[
|V(H_{r,q})| = rq,
\qquad
|E(H_{r,q})| = \frac{q^2+2q-1}{2}.
\]
Equivalently, for $n=rq$,
\[
|E(H_{r,q})|
=
\left(\frac{1}{2r^2}+o(1)\right)n^2
\qquad (q\to\infty,\ r\text{ fixed}).
\]
\end{theorem}

Theorem \ref{thm:main} gives the first unified construction showing that $\mathrm{ex}_{\mathrm{lin}}(n,G_{r\times r}) = \Theta_{r}(n^2)$ for all $r \geq 3$. We discuss this construction in Section \ref{sec:construction}, and then complete the proof of Theorem \ref{thm:main} in Section \ref{sec:CB}. 

It is important to emphasize that for fixed $r\ge 4$, the constructions of F\"uredi--Ruszink\'o~\cite{FurediRuszinko} and Glock-Joos-Kim-K{\"u}hn-Lichev \cite{conflictfree} come much closer to $n^2/r(r-1)$, and at the moment we do not know how to modify our construction to reach that edge density. While we can improve the constant $1/2r^{2}$ by a little bit, we choose to not do this at the moment in order to maximize the clarity of the exposition. Instead, we would like to emphasize that the same Cayley--Bacharach viewpoint can be used to naturally rule out more general patterns. For example, one can define the $T$-hole punctured intersection $P_{r}(T)$ by starting from the $r\times r$ grid incidence pattern and deleting a set $T$ of the $r^2$ row--column intersection vertices (replacing them with private filler vertices to keep all $2r$ edges $r$-uniform). For convenience, let us also write $P_{r,t}$ for the set of all such patterns $P_{r}(T)$ going over all sets $T$ of size $t$. For example, for $r=3$ and $t=1$, the family $P_{3,1}$ consists of all $3\times 3$ intersection pattern $P_{3}(T)$ with one missing
row--column incidence, made $3$-uniform by adding two private/auxiliary vertices.

\begin{figure}[ht]
\centering
\begin{tikzpicture}[
  scale=1,
  every node/.style={circle,draw,inner sep=1.2pt},
  lab/.style={draw=none,circle=none,inner sep=0pt}
]
\foreach \i in {1,2,3}{
  \foreach \j in {1,2,3}{
    \ifnum\i=3\relax
      \ifnum\j=3\relax
      \else
        \node (v\i\j) at (\j,-\i) {$v_{\i\j}$};
      \fi
    \else
      \node (v\i\j) at (\j,-\i) {$v_{\i\j}$};
    \fi
  }
}

\node[rectangle,draw,inner sep=2pt] at (3,-3) {$\times$};

\node (a33) at (4.15,-3) {$a_{33}$};
\node (b33) at (3,-4.15) {$b_{33}$};

\tikzset{seg/.style={thick,gray!55,shorten <=2.5pt,shorten >=2.5pt}}

\node[lab] at (0.35,-1) {$R_{1}$};
\node[lab] at (0.35,-2) {$R_{2}$};
\node[lab] at (0.35,-3) {$R_{3}$};

\draw[seg] (v11) -- (v12);
\draw[seg] (v12) -- (v13);
\draw[seg] (v21) -- (v22);
\draw[seg] (v22) -- (v23);
\draw[seg] (v31) -- (v32);
\draw[seg] (v32) -- (a33);

\node[lab] at (1,-0.25) {$C_{1}$};
\node[lab] at (2,-0.25) {$C_{2}$};
\node[lab] at (3,-0.25) {$C_{3}$};

\draw[seg] (v11) -- (v21);
\draw[seg] (v21) -- (v31);
\draw[seg] (v12) -- (v22);
\draw[seg] (v22) -- (v32);
\draw[seg] (v13) -- (v23);
\draw[seg] (v23) -- (b33);
\end{tikzpicture}
\caption{The punctured intersection $P_3(T)$ with $T=\{(3,3)\}$: the intersection $R_3\cap C_3$ is removed and replaced by private vertices $a_{33}\in R_3$ and $b_{33}\in C_3$.}
\label{fig:P31}
\end{figure}

We can use our method to show that one can also always construct dense $r$-uniform linear hypergraphs avoiding such configurations, as well.

\begin{theorem}\label{thm:punctured}
For every $r \geq 3$ and every sufficiently large odd prime power $q$, there exists an $r$-uniform linear
hypergraph $H^{\parallel}_{r,q}$ with
\[
|V(H^{\parallel}_{r,q})| = rq,
\qquad
|E(H^{\parallel}_{r,q})| = q^2 = \frac{1}{r^2}|V|^2,
\]
and such that $H^{\parallel}_{r,q}$ contains no $t$-hole punctured intersection from $P_{r,t}$, for every $t\le r-2$. 
\end{theorem}

We discuss this in more detail in Section \ref{sec:punctured} (as well as the proof).

As a final comment, it is perhaps instructive to contrast Theorem~\ref{thm:punctured} (already the case $r=3$ and $t=1$ above) with the behavior of other natural local obstructions in linear $3$-uniform hypergraphs. For example, consider the linear $3$-uniform hypergraph $W_{3 \times 3}$ on $9$ vertices $\left\{v_{ij}:1\le i,j\le 3\right\}$ with edges given by the three rows $\left\{v_{i1},v_{i2},v_{i3}\right\}$ for $i=1,2,3$, together with two of the three columns $\left\{v_{1j},v_{2j},v_{3j}\right\}$ (equivalently, $W_{3 \times 3}$ is a $3 \times 3$ grid with one column removed, sometimes also called a \emph{wicket}). In \cite{SolymosiWickets}, Solymosi showed that forbidding such configurations is substantially more restrictive than forbidding $G_{3 \times 3}$ or $P_{3,1}$. In particular, wicket-free linear $3$-uniform hypergraphs on $n$ vertices must always have $o(n^2)$ edges. 

\medskip

\noindent {\bf{Acknowledgments.}} The author would like to thank Jozsef Solymosi for useful discussions and comments on an earlier version of the manuscript. 

\section{Construction} \label{sec:construction}

Fix $r\ge 3$ and let $q$ be an odd prime power.
We work in the affine plane $\AG(2,q)=\F_q^2$ (for now).

Choose $r-1$ distinct nonsquares $\alpha_1,\dots,\alpha_{r-1}\in \F_q^\times$.
Define horizontal lines
\[
L_t := \{(x,\alpha_t): x\in \F_q\}\qquad(1\le t\le r-1),
\qquad
A := \bigcup_{t=1}^{r-1} L_t.
\]
Let
\[
B := \{(x,x^2): x\in \F_q\}.
\]

We record some easy observations.

\begin{fact}\label{lem:disjoint}
$A\cap B=\varnothing$.
\end{fact}

\begin{proof}
If $(x,\alpha_t)\in A\cap B$, then $\alpha_t=x^2$ is a square in $\F_q^\times$, contradicting that $\alpha_t$ is a nonsquare.
\end{proof}

Moreover, since $B$ is a conic, we also know that:

\begin{fact}\label{lem:nothree}
No affine line meets $B$ in three distinct points.
\end{fact}

We now define the hypergraph $H_{r,q}$. Let
\[
V := A\dot\cup B.
\]
For each \emph{nonhorizontal} affine line $\ell$ with $\ell\cap B\neq\varnothing$, choose one point
$p(\ell)\in \ell\cap B$ (if there are two choices, fix any deterministic rule).
Define an edge
\[
e(\ell) := (\ell\cap A)\ \cup\ \{p(\ell)\}.
\]
Let
\[
E:=\{e(\ell): \ell \text{ nonhorizontal and } \ell\cap B\neq\varnothing\}
\]
and set $H_{r,q}:=(V,E)$.

Before $G_{r \times r}$-freeness, we record the most important property that $H$ satisfies. 

\begin{fact}\label{prop:uniformlinear}
$H_{r,q}$ is an $r$-uniform linear hypergraph.
\end{fact}

\begin{proof}
Fix a nonhorizontal line $\ell$. Since $\ell$ is not horizontal, it meets each horizontal line $L_t$ in exactly one
point, so $|\ell\cap A|=r-1$. By construction we add exactly one point $p(\ell)\in B$, hence $|e(\ell)|=r$. For linearity, note that each edge $e(\ell)$ is contained in its supporting line $\ell$.
Two distinct affine lines intersect in at most one point, hence two distinct edges intersect in at most one vertex.
\end{proof}

Last but not least, let us record the vertex/edge statistics of our $r$-uniform hypergraph $H_{r,q}$.

\begin{fact}
$|V(H_{r,q})| = |A|+|B| = (r-1)q + q = rq$.
\end{fact}

\begin{fact}\label{prop:Ecount}
For odd $q$,
\[
|E(H_{r,q})| = q + (q-1)\cdot\frac{q+1}{2} = \frac{q^2+2q-1}{2}
= \left(\frac12+o(1)\right)q^2.
\]
\end{fact}

\begin{proof}
There are $q$ vertical lines $x=c$, and each meets $B$ in $(c,c^2)$.

A nonvertical nonhorizontal line has equation $y=mx+b$ with $m\in\F_q^\times$ and $b\in\F_q$.
It meets $B$ iff $x^2=mx+b$ has a solution, i.e.\ iff $x^2-mx-b=0$ has a solution.
Its discriminant is $\Delta=m^2+4b$. As $b$ varies, $\Delta$ ranges uniformly over $\F_q$.
Over an odd field, exactly $(q+1)/2$ elements are squares (including $0$), so for each fixed $m\neq 0$ exactly
$(q+1)/2$ values of $b$ yield $\ell\cap B\neq\varnothing$.
Summing gives $(q-1)(q+1)/2$ such lines plus the $q$ vertical lines.
\end{proof}

\section{Proof of Theorem \ref{thm:main}} \label{sec:CB}

We start by stating the main driving force behind the construction. 

\begin{lemma}[Cayley-Bacharach]\label{cor:planeCB}
Let $D_1,D_2\subset \PP^2$ be plane curves of degrees $r$ and $r$ with no common component.
Assume $D_1\cap D_2$ consists of exactly $r^2$ distinct points $X=\{x_1,\dots,x_{r^2}\}$.
Then any homogeneous polynomial $f$ of degree $\le 2r-3$ that vanishes on $r^2-1$ points of $X$
must vanish on all $r^2$ points of $X$.
\end{lemma}

This is the 19th century Cayley-Bacharach theorem \cite{Eisenbud}, which in the case $r=3$ is already remarkably powerful: given two cubics which meet in nine points, any (third) cubic through eight of those nine must automatically pass through the ninth. As already mentioned in Section \ref{sec:intro}, this fact immediately recovers Pascal's theorem (and, when the conic degenerates, Pappus's theorem), however it has numerous other applications as well. A more recent example is the celebrated work of Green and Tao \cite{GreenTao} on the Dirac–Motzkin ordinary-line conjecture, where it was used as a mechanism to derive global structure in planar configurations from local information.

To complete the proof of Theorem \ref{thm:main}, we now use Lemma \ref{cor:planeCB} to establish the $G_{r \times r}$-freeness of our linear $r$-uniform hypergraph $H_{r,q}$. 

\begin{proof}[Proof of Theorem \ref{thm:main}]
Assume for contradiction that $H_{r,q}$ contains a subhypergraph isomorphic to $G_{r\times r}$.
Let $R_1,\dots,R_r$ be the row edges and $C_1,\dots,C_r$ the column edges.

By construction, every edge of $H_{r,q}$ is supported on a unique nonhorizontal affine line.
Thus there exist affine lines $\ell_1,\dots,\ell_r$ and $m_1,\dots,m_r$ such that
\[
R_i=e(\ell_i)\quad(1\le i\le r),
\qquad
C_j=e(m_j)\quad(1\le j\le r).
\]
No $\ell_i$ equals any $m_j$: otherwise $R_i=C_j$ would have $r$ common vertices, contradicting $|R_i\cap C_j|=1$.

Because $|R_i\cap C_j|=1$, the lines $\ell_i$ and $m_j$ intersect in the affine plane (so they are not parallel),
and their intersection point is exactly the grid vertex $R_i\cap C_j$.
Since the grid has $r^2$ distinct vertices, the $r^2$ points $\ell_i\cap m_j$ are all distinct. Let $\overline{\ell_i},\overline{m_j}\subset \PP^2$ be projective closures and define degree-$r$ curves
\[
D_1 := \overline{\ell_1}\cup\cdots\cup \overline{\ell_r},
\qquad
D_2 := \overline{m_1}\cup\cdots\cup \overline{m_r}.
\]
Then $D_1\cap D_2$ consists of exactly the $r^2$ distinct points $\overline{\ell_i}\cap \overline{m_j}$
(all affine), hence Lemma ~\ref{cor:planeCB} applies to $D_1,D_2$.

Every edge of $H_{r,q}$ contains exactly one vertex in $B$.
The grid has $2r$ edges, hence $2r$ incidences of the form ``(edge, its $B$-vertex)''.
In the grid, every vertex lies in exactly two edges (one row, one column), so each $B$-vertex is counted twice.
Therefore the grid contains exactly $(2r)/2=r$ distinct vertices on $B$.
Label these points $b_1,\dots,b_r\in B$.
All remaining $r^2-r=r(r-1)$ grid vertices lie in $A$. Define the curve
\[
D \;:=\; A\ \cup\ \overline{b_1b_2}\ \cup\ \overline{b_2b_3}\ \cup\ \cdots\ \cup\ \overline{b_{r-2}b_{r-1}},
\]
a union of $(r-1)$ lines (those forming $A$) and $(r-2)$ additional lines, hence its degree satisfies $\deg(D)=2r-3$.
Let $f$ be a homogeneous polynomial of degree $2r-3$ defining $D$.

By construction:
\begin{itemize}[leftmargin=2em]
\item $D$ contains every grid vertex lying in $A$, since $A\subset D$.
\item $D$ contains $b_1,\dots,b_{r-1}$ (each is an endpoint of one of the lines $\overline{b_ib_{i+1}}$).
\item $b_r\notin D$: indeed $b_r\notin A$ by Fact~\ref{lem:disjoint}, and if $b_r$ lay on
$\overline{b_ib_{i+1}}$ then that line would meet the conic $B$ in the three distinct points
$b_i,b_{i+1},b_r$, contradicting Fact~\ref{lem:nothree}.
\end{itemize}
Hence $f$ vanishes on all grid vertices except $b_r$. In other words, if $X=D_1\cap D_2$ is the set of $r^2$ grid vertices, then $f$ vanishes on $r^2-1$ points of $X$
but not on the last point $b_r$.

Since $\deg f = 2r-3$, Lemma~\ref{cor:planeCB} forces $f$ to vanish on \emph{all} points of $X$,
in particular on $b_r$ as well, contradiction. Therefore $H_{r,q}$ contains no $r\times r$ grid.
\end{proof}

\section{Punctured intersections}\label{sec:punctured}

In the proof of Theorem~\ref{thm:main} we ruled out a full $r\times r$ grid by the following template:
a hypothetical grid produces two degree-$r$ curves $D_1,D_2$ that are unions of $r$ lines, hence
a complete intersection $X=D_1\cap D_2$ of size $r^2$; we then explicitly build a curve of degree $2r-3$
through $r^2-1$ points of $X$, contradicting Cayley--Bacharach.

The same mechanism also forbids certain \emph{almost} complete intersections, obtained by deleting a small number of points from $X$.
Combinatorially, this leads to a natural family of ``punctured grids,'' which we now define.
The point of the definition is that we remove some row--column incidences, but we keep the hypergraph $r$-uniform and linear by
adding \emph{private filler vertices}.

\medskip

\begin{definition}\label{def:Prt}
Fix $r\ge 3$ and let $T\subseteq [r]\times[r]$ with $|T|=t$.
We define an $r$-uniform linear hypergraph $P_r(T)$ with $2r$ edges
\[
R_1,\dots,R_r \qquad\text{and}\qquad C_1,\dots,C_r,
\]
thought of as $r$ ``rows'' and $r$ ``columns,'' as follows:

\smallskip

\noindent\textbf{(i) The non-holes.}
For each pair $(i,j)\notin T$ introduce a vertex $v_{ij}$ and declare that it lies in the intersection
\[
v_{ij}\in R_i\cap C_j.
\]
Thus, away from $T$, the pattern agrees with the usual $r\times r$ grid.

\smallskip

\noindent\textbf{(ii) The holes.}
For each $(i,j)\in T$ we delete the row--column incidence by forcing
\[
R_i\cap C_j=\varnothing.
\]
To keep $R_i$ and $C_j$ $r$-uniform, we introduce two \emph{private} vertices $a_{ij}$ and $b_{ij}$ and place them as
\[
a_{ij}\in R_i,\qquad b_{ij}\in C_j,
\]
with $a_{ij}$ contained in no other edge and $b_{ij}$ contained in no other edge.

\smallskip

Equivalently, the edges are
\[
R_i=\{v_{ij}:(i,j)\notin T\}\cup\{a_{ij}:(i,j)\in T\},\qquad
C_j=\{v_{ij}:(i,j)\notin T\}\cup\{b_{ij}:(i,j)\in T\}.
\]

\smallskip

Then $P_r(T)$ is $r$-uniform and linear, has $2r$ edges and $r^2+t$ vertices.
We write $P_{r,t}$ for the family $\{P_r(T): |T|=t\}$.
Note that $P_{r,0}=G_{r\times r}$.
\end{definition}

\subsection{A degree budget lemma}

We will only use punctured intersections in the natural ``geometric'' regime where the underlying line configuration
really does form a complete intersection of size $r^2$ in $\PP^2$.

\begin{definition}\label{def:transverse}
Let $\ell_1,\dots,\ell_r$ and $m_1,\dots,m_r$ be affine lines in $\AG(2,q)$, and let $\overline{\ell_i},\overline{m_j}\subset\PP^2$
be their projective closures. We say that the pair of families is \emph{transverse} if the $r^2$ points
\[
x_{ij}:=\overline{\ell_i}\cap\overline{m_j}\qquad (1\le i,j\le r)
\]
are all distinct in $\PP^2$ (equivalently, the degree-$r$ curves $D_1:=\bigcup_{i=1}^r\overline{\ell_i}$ and
$D_2:=\bigcup_{j=1}^r\overline{m_j}$ meet in exactly $r^2$ distinct points).
\end{definition}

The next lemma is the abstract ``degree budget'' principle behind the proof of Theorem \ref{thm:main} above.

\begin{lemma}\label{lem:degreebudget}
Fix $r\ge 3$.
Let $D_1,D_2\subset\PP^2$ be plane curves of degrees $r$ and $r$ with no common component, such that
\[
X:=D_1\cap D_2=\{x_1,\dots,x_{r^2}\}
\]
consists of exactly $r^2$ distinct points.
Let $C\subset\PP^2$ be a plane curve of degree $d$.
Suppose that exactly $t$ points of $X$ lie outside $C$, i.e.
\[
|X\setminus C|=t\ge 1.
\]
If
\begin{equation}\label{eq:degreebudget}
d+(t-1)\le 2r-3,
\end{equation}
the curve $C$ cannot contain $r^2- t$ points of $X$ while missing $t\ge 1$ points.
\end{lemma}

Like before, the proof is a simple application of Lemma \ref{cor:planeCB}. 

\begin{proof}
Write $Y:=X\setminus C=\{y_1,\dots,y_t\}$.
Fix one point, say $y_t$, that we will \emph{not} cover.

For each $k\in\{1,\dots,t-1\}$ we choose a projective line $s_k\subset\PP^2$ such that
\[
y_k\in s_k
\qquad\text{and}\qquad
y_t\notin s_k.
\]
Such a choice is always possible: there are many lines through $y_k$, and only one line passes through both $y_k$ and $y_t$,
so we choose any other line through $y_k$.

Define the algebraic curve
\[
E \;:=\; C \cup s_1\cup\cdots\cup s_{t-1}.
\]
Then $\deg(E)\le d+(t-1)$ (with equality unless some $s_k$ is already a component of $C$ or repeats another $s_\ell$,
which can only decrease the degree).

By construction, $E$ contains:
\begin{itemize}[leftmargin=2em]
\item all points of $X\cap C$ (since $C\subset E$),
\item and the points $y_1,\dots,y_{t-1}$ (since $y_k\in s_k\subset E$ for $k\le t-1$),
\end{itemize}
so $E$ contains all points of $X$ except possibly $y_t$.
On the other hand, $y_t\notin E$ because $y_t\notin C$ (by definition of $Y$) and $y_t\notin s_k$ for all $k\le t-1$.

Thus $E$ is a curve of degree at most $d+(t-1)$ that contains $r^2-1$ points of $X$ but misses the last point $y_t$.
If \eqref{eq:degreebudget} holds, then $\deg(E)\le 2r-3$, and Lemma ~\ref{cor:planeCB} forces $E$ to contain \emph{all} points of $X$, contradiction.
\end{proof}

\subsection{Application to the construction $H_{r,q}$ from Theorem~\ref{thm:main}}

We now explain how the same degree-budget lemma shows that our grid-free construction also forbids
punctured intersections (in the natural transverse ``punctured complete intersection'' sense).

Let $H_{r,q}$ be the hypergraph from Section~\ref{sec:construction}, with vertex set
\[
V(H_{r,q})=A\dot\cup B,
\]
where $A$ is the union of the $r-1$ horizontal lines and $B=\{(x,x^2):x\in\F_q\}$ is the parabola.

Let $C_0\subset\PP^2$ be the projective closure of $A\cup B$.
Then $C_0$ is a reducible plane curve of degree
\[
\deg(C_0)=\deg(A)+\deg(B)=(r-1)+2=r+1,
\]
and every vertex of $H_{r,q}$ lies on $C_0$.

\begin{corollary}\label{cor:Prt-in-main}
Fix $r\ge 4$ and let $H_{r,q}$ be the hypergraph from Theorem~\ref{thm:main}.
Let $T\subseteq [r]\times[r]$ with $|T|=t\le r-3$.
Then $H_{r,q}$ contains no transverse copy of $P_r(T)$ that is realized by $r$ row lines and $r$ column lines
whose $t$ ``missing'' row--column intersection points are not vertices of $H_{r,q}$.
\end{corollary}

\begin{proof}
Suppose for contradiction that such a copy exists.
By Definition~\ref{def:transverse} we obtain two degree-$r$ curves
$D_1,D_2\subset\PP^2$ which are unions of $r$ lines, such that
\[
X:=D_1\cap D_2=\{x_{ij}:1\le i,j\le r\}
\]
consists of exactly $r^2$ distinct points.

By hypothesis, for every $(i,j)\notin T$ the row and column edges intersect, hence the intersection point $x_{ij}$ is a
vertex of $H_{r,q}$ and therefore lies on $C_0$.
For every $(i,j)\in T$, the corresponding row and column edges are disjoint \emph{and} the missing incidence is realized
by an intersection point $x_{ij}\notin V(H_{r,q})$, hence $x_{ij}\notin C_0$.
Thus $C_0$ contains exactly $r^2-t$ points of $X$ and misses exactly $t$ points.

Since $\deg(C_0)=r+1$ and $t\le r-3$, we have
\[
\deg(C_0)+(t-1)=(r+1)+(t-1)\le (r+1)+(r-4)=2r-3.
\]
Applying Lemma~\ref{lem:degreebudget} with $C=C_0$ gives a contradiction.
\end{proof}

\begin{remark}[A small caveat]\label{rem:caveat}
In $H_{r,q}$ each hyperedge is supported on a line.
If a line $\ell$ meets the parabola $B$ in two points, we choose only \emph{one} of them to lie in the edge $e(\ell)$.
Because of this, it can happen that two supporting lines intersect at a point of $B$ while the corresponding
hyperedges are disjoint (if at least one of the two edges chose the \emph{other} intersection point with $B$).
Lemma~\ref{lem:degreebudget} and Corollary~\ref{cor:Prt-in-main} therefore address punctured intersections only in the
literal ``punctured complete intersection'' sense, where the missing row--column incidences correspond to
line-intersection points that lie outside the vertex set (equivalently, outside the ambient curve $C_0$).

In contrast, in the parallel-line model below the edges are \emph{full} line--vertex intersections, so this subtlety
does not arise there.
\end{remark}

\subsection{A denser model: transversals of $r$ parallel lines}

We now record the closely related (and simpler) construction from Theorem \ref{thm:punctured}, where Cayley--Bacharach will similarly allow us to forbid \emph{all} punctured intersections $P_{r,t}$ up to the natural degree threshold.

\begin{proposition}\label{prop:Prt-free}
Fix $r\ge 3$ and a prime power $q$.
Let $L_1,\dots,L_r\subset \AG(2,q)$ be $r$ distinct parallel affine lines, and set
\[
V:=\bigcup_{i=1}^r L_i,
\qquad |V|=rq.
\]
Define an $r$-uniform hypergraph $H^{\parallel}_{r,q}$ on $V$ whose edges are
\[
e(\ell):=\ell\cap V,
\]
as $\ell$ ranges over all affine lines not parallel to the $L_i$.
Then $H^{\parallel}_{r,q}$ is $r$-uniform, linear, and
\[
|E(H^{\parallel}_{r,q})| = q^2 = \frac{1}{r^2}|V|^2.
\]
Moreover, for every $t\le r-2$, the hypergraph $H^{\parallel}_{r,q}$ contains no copy of any $P_r(T)\in P_{r,t}$
whose row/column supporting lines form a transverse configuration in the sense of Definition~\ref{def:transverse}.
\end{proposition}

\begin{proof}
Every affine line $\ell$ not parallel to the $L_i$ meets each $L_i$ in exactly one point, hence $|e(\ell)|=r$.
Two distinct affine lines intersect in at most one point, so two distinct edges intersect in at most one vertex.

To check the edge count, fix the common direction of the $L_i$.
For each of the $q$ possible slopes not equal to this direction and each of the $q$ possible intercepts in that slope
class, we get a distinct affine line not parallel to the $L_i$.
Thus there are exactly $q^2$ such lines, hence $|E(H^{\parallel}_{r,q})|=q^2$.

Finally, we check $P_{r,t}$-freeness, by using Lemma \ref{lem:degreebudget}.Let $C\subset\PP^2$ be the projective closure of $L_1\cup\cdots\cup L_r$.
Then $C$ is a union of $r$ lines, so $\deg(C)=r$, and $V\subset C$.

Suppose for contradiction that $H^{\parallel}_{r,q}$ contains a copy of some $P_r(T)\in P_{r,t}$ with $t\le r-2$,
realized by transverse row and column line families.
Let $\ell_1,\dots,\ell_r$ and $m_1,\dots,m_r$ be the supporting affine lines of the $r$ rows and $r$ columns,
and let $D_1:=\bigcup_{i=1}^r\overline{\ell_i}$ and $D_2:=\bigcup_{j=1}^r\overline{m_j}$.
Transversality means that $X:=D_1\cap D_2$ consists of exactly $r^2$ distinct points $x_{ij}=\overline{\ell_i}\cap\overline{m_j}$.

Because edges are full intersections with $V$, we have:
\begin{itemize}[leftmargin=2em]
\item If $(i,j)\notin T$, then the row edge $e(\ell_i)$ and column edge $e(m_j)$ intersect in a vertex of $V$,
so $x_{ij}\in V\subset C$.
\item If $(i,j)\in T$, then $e(\ell_i)\cap e(m_j)=\varnothing$, hence $x_{ij}\notin V$.
Since $V$ consists of \emph{all} affine points on $C$, this forces $x_{ij}\notin C$.
\end{itemize}
Thus $C$ contains exactly $r^2-t$ points of $X$ and misses exactly $t$ points.
Since $\deg(C)=r$ and $t\le r-2$, we have
\[
\deg(C)+(t-1)=r+(t-1)\le r+(r-3)=2r-3.
\]
Lemma~\ref{lem:degreebudget} gives a contradiction.
\end{proof}

\section{Concluding remarks}\label{sec:concluding}

In this note, we primarly wanted to highlight a robust principle in finite geometry: once a forbidden configuration can be viewed
as (almost) a \emph{zero-dimensional complete intersection}, Cayley--Bacharach prevents one from
``removing a single intersection point'' using a low-degree auxiliary curve/hypersurface.
In Sections~\ref{sec:CB}--\ref{sec:punctured} this philosophy yields a very short certificate of
grid-freeness (and, more generally, punctured intersection-freeness) for a natural line-based
construction in $\AG(2,q)$.

We believe this viewpoint should be useful well beyond the planar grid problem. This comes from the fact that Lemma~\ref{cor:planeCB} admits the following higher-dimensional generalization.

\begin{theorem}\label{thm:KarasevCB}
Let $g_1,\dots,g_n$ be homogeneous polynomials on $\PP^n$ of degrees $k_1,\dots,k_n$.
Assume the system
\[
g_1(x)=\cdots=g_n(x)=0
\]
has exactly $k:=k_1k_2\cdots k_n$ \emph{isolated} solutions $X=\{x_1,\dots,x_k\}$. Then there exist nonzero coefficients $\alpha_1,\dots,\alpha_k$ such that for every homogeneous polynomial $f$
with
\[
\deg f \le \sum_{i=1}^n k_i - n - 1,
\]
one has the linear relation
\[
\sum_{i=1}^k \alpha_i\, f(x_i) = 0.
\]
In particular, $f$ vanishes on all of $X$ if and only if
it vanishes on all but one point of $X$.
\end{theorem}

See for example Karasev~\cite[Theorem 5.1]{Karasev} and the references therein. Note that this recovers Lemma~\ref{cor:planeCB} when $n=2$ and $k_1=k_2=r$. Like Lemma~\ref{cor:planeCB}, Theorem~\ref{thm:KarasevCB} also works over arbitrary fields, provided all the points of $X$ are defined over that field. A particularly nice consequence of Lemma~\ref{thm:KarasevCB} is the following classical theorem of Alon and F\"uredi~\cite{AlonFuredi}.

\begin{theorem}\label{thm:AlonFuredi}
Let $\F$ be a field and let $A_1,\dots,A_d\subset \F$ be finite nonempty sets.
Write $A:=A_1\times\cdots\times A_d$.
If $f\in \F[x_1,\dots,x_d]$ is a polynomial of total degree
\[
\deg(f)\le \sum_{i=1}^d (|A_i|-1)-1,
\]
and $f$ vanishes on all but at most one point of $A$, then $f$ vanishes on all of $A$.
\end{theorem}
We refer the reader to a blog post by the author \cite{PohoataBlog} for a detailed discussion of this connection.

One may use Theorem~\ref{thm:KarasevCB} to look for other extremal constructions which are manifestations of the same Cayley--Bacharach ``completion'' mechanism. In fact, some other ones already exist, in some sense. For example, consider the so-called \emph{Erd\H{o}s box problem}, i.e.\ the question of estimating the $r$-uniform Tur\'an number $\mathrm{ex}_r(n, K^{(r)}_{2,\dots,2})$ of the complete $r$-partite $r$-uniform hypergraph with two vertices in each part. The best known upper bound, going back to Erd\H{o}s \cite{Erdos64}, is $\mathrm{ex}_r(n, K^{(r)}_{2,\dots,2}) = O\left(n^{\,r - 1/2^{r-1}}\right)$.
On the lower bound side, Conlon--Pohoata--Zakharov~\cite{CPZBox} showed more recently that $\mathrm{ex}_r(n, K^{(r)}_{2,\dots,2}) = \Omega\left( n^{r - \lceil \frac{2^r-1}{r}\rceil^{-1}} \right)$, which remains the best known general lower bound for all $r \geq 2$. More recently, Gordeev~\cite{Gordeev} observed that the so-called Combinatorial Nullstellensatz (famously introduced by Alon in \cite{Alon1999}) provides a clean framework for constructing $K^{(r)}_{s_1,\dots,s_r}$-free $r$-graphs, by considering zero sets of $r$-variate polynomials where the maximal monomial has the form $x_{1}^{s_{1}-1}\ldots x_{r}^{s_{r}-1}$. In particular, this led to a short alternative construction showing that $\mathrm{ex}_r(n, K^{(r)}_{2,\dots,2})=\Omega(n^{\,r-1/r})$ along an infinite sequence of~$n$. In hindsight, it turns out that Gordeev's construction is, at the core, yet another manifestation of the Cayley--Bacharach completion phenomenon from the present paper: a copy of $K^{(r)}_{2,\dots,2}$ is a $2\times\cdots\times 2$ box, which is itself the complete intersection of the $r$ coordinate quadrics in $\PP^r$. On the other hand, Theorem~\ref{thm:AlonFuredi}, itself a consequence of the higher-dimensional Cayley--Bacharach Theorem~\ref{thm:KarasevCB}, is precisely the statement that no polynomial of insufficiently high degree can separate a single vertex of the box from the remaining $2^r-1$. 

It would be very interesting to push this unifying perspective further and determine whether other old and new Tur\'an-type lower bounds in extremal combinatorics can be systematically derived by exploiting the Cayley--Bacharach completion obstruction presented in this paper.



\begin{thebibliography}{99}

\bibitem{Alon1999}
N.~Alon,
\emph{Combinatorial Nullstellensatz},
Combin. Probab. Comput. \textbf{8} (1999), no.~1--2, 7--29.

\bibitem{AlonFuredi}
N.~Alon and Z.~F\"uredi,
\emph{Covering the cube by affine hyperplanes},
European J. Combin. \textbf{14} (1993), no.~2, 79--83.

\bibitem{CPZBox}
D.~Conlon, C.~Pohoata, and D.~Zakharov,
\newblock \emph{Random multilinear maps and the Erd\H{o}s box problem},
\newblock Discrete Analysis \textbf{2021:17} (2021), 8 pp.

\bibitem{DP} M.~Delcourt and L.~Postle,
\emph{Proof of the High Girth Existence Conjecture via Refined Absorption}, 
\newblock arXiv:2402.17856.

\bibitem{Eisenbud}
D.~Eisenbud, M.~Green, and J.~Harris,
\emph{Cayley--Bacharach theorems and conjectures},
Bull. Amer. Math. Soc. \textbf{33} (1996), no.~3, 295--324.

\bibitem{Erdos64}
P.~Erd\H{o}s,
\emph{On extremal problems of graphs and generalized hypergraphs}, Israel J. Math. \textbf{2} (1964), 183--190.

\bibitem{FurediRuszinko}
Z.~F\"uredi and M.~Ruszink\'o,
\newblock \emph{Uniform hypergraphs containing no grids},
\newblock Adv.\ Math.\ \textbf{240} (2013), 302--324.
\newblock arXiv:1103.1691.

\bibitem{GishShapira}
L.~Gishboliner and A.~Shapira,
\newblock \emph{Constructing dense grid-free linear $3$-graphs},
\newblock Proc.\ Amer.\ Math.\ Soc.\ \textbf{150} (2022), 69--74.
\newblock arXiv:2010.14469.

\bibitem{conflictfree}
S. Glock, F.~Joos, J.~Kim, M.~K{\"u}hn, and L.~Lichev,
\emph{Conflict-free hypergraph matchings},
J. Lond. Math. Soc. \textbf{109} (2024), e12899.

\bibitem{Gordeev}
A.~Gordeev,
\emph{Combinatorial Nullstellensatz and Turán numbers of complete $r$-partite $r$-uniform hypergraphs},
Discrete Math. \textbf{347} (2024).

\bibitem{GreenTao}
B.~Green and T.~Tao,
\newblock \emph{On sets defining few ordinary lines},
\newblock Discrete Comput.\ Geom.\ \textbf{50} (2013), no.~2, 409--468.

\bibitem{Karasev}
R.~Karasev,
\newblock \emph{Residues and the Combinatorial Nullstellensatz},
\newblock arXiv:1503.08004.

\bibitem{HighGirth}
M.~Kwan, A.~Sah, M.~Sawhney, and M.~Simkin,
\emph{High-girth Steiner triple systems},
Ann. of Math. \textbf{200} (2024), no.~3, 1059--1156.

\bibitem{PohoataBlog}
C. Pohoata,
\newblock \emph{The Cayley-Bacharach and its applications},
\newblock blog post (2025).


\bibitem{Solymosi}
J.~Solymosi,
\newblock \emph{On the Tur\'an number of the $G_{3\times 3}$ in linear hypergraphs},
\newblock arXiv:2504.16973.

\bibitem{SolymosiWickets}
J.~Solymosi,
\newblock \emph{Wickets in $3$-uniform hypergraphs},
\newblock arXiv:2305.01193.


\bibitem{TaoBlog}
T.~Tao,
\newblock \emph{Pappus's theorem and elliptic curves},
\newblock blog post (2011).

\end{thebibliography}
\end{document}